\documentclass[12pt]{amsart}
\numberwithin{equation}{section}

\usepackage{amssymb}
\def\cb{{\mathcal B}}

\def\ch{{\mathcal H}}

\def\ck{{\mathcal K}}

\def\ga{{\mathfrak A}}

\def\a{\alpha}
\def\b{\beta}
\def\g{\gamma}

\def\eps{\varepsilon}

\def\r{\rho}
\def\s{\sigma} 
\def\t{\tau}

\def\om{\omega} \def\Om{\Omega}

\newtheorem{thm}{Theorem}[section]
\newtheorem{lem}[thm]{Lemma}

\newtheorem{prop}[thm]{Proposition}

\def\aut{\mathop{\rm Aut}}

\def\ad{\mathop{\rm ad}}

\newcommand{\nn}{\nonumber}

\def\tr{\mathop{\rm Tr}}

\begin{document}

\title[entangled ergodic theorem]
{the entangled ergodic theorem in the almost periodic case}
\author{Francesco Fidaleo}
\address{Francesco Fidaleo,
Dipartimento di Matematica,
Universit\`{a} di Roma Tor Vergata, 
Via della Ricerca Scientifica 1, Roma 00133, Italy} \email{{\tt
fidaleo@mat.uniroma2.it}}


\begin{abstract}
Let $U$ be a unitary operator acting on the Hilbert space $\ch$, and 
$\a:\{1,\dots, 2k\}\mapsto\{1,\dots, k\}$ a pair partition.
Then the ergodic average 
$$
\frac{1}{N^{k}}\sum_{n_{1},\dots,n_{k}=0}^{N-1}
U^{n_{\a(1)}}A_{1}U^{n_{\a(2)}}\cdots 
U^{n_{\a(2k-1)}}A_{2k-1}U^{n_{\a(2k)}}
$$
converges in the strong
operator topology provided $U$ is almost periodic, that is when 
$\ch$ is generated by the eigenvalues of $U$. We apply the present result to obtain the convergence of the Cesaro mean of several multiple correlations. 
\end{abstract}

\maketitle

\section{introduction}

An entangled ergodic theorem was introduced in \cite{AHO} in 
connection with the quantum central limit theorem, and clearly 
formulated in \cite{L}.
Namely, let $U$ be 
a unitary operator on the Hilbert space $\ch$, and for $m\geq k$, 
$\a:\{1,\dots,m\}\mapsto\{1,\dots, k\}$ a partition of the 
set $\{1,\dots,m\}$ in $k$ parts.
The {\it entangled ergodic theorem} concerns 
the convergence in the strong, or merely weak (s--limit, or w--limit 
for short) operator topology, of the multiple Cesaro mean
\begin{equation}
\label{0}
\frac{1}{N^{k}}\sum_{n_{1},\dots,n_{k}=0}^{N-1}
U^{n_{\a(1)}}A_{1}U^{n_{\a(2)}}\cdots 
U^{n_{\a(m-1)}}A_{m-1}U^{n_{\a(m)}}\,,
\end{equation}
$A_{1},\dots,A_{m-1}$ being bounded operators acting on $\ch$.

Expressions like \eqref{0} naturally appear in the study of multiple 
correlations, see Section \ref{coorre} below. The simplest case is nothing but the well known mean 
ergodic theorem 
due to John von Neumann 
\begin{equation}
\label{jvn}
\mathop{\rm s\!-\!lim}_{N}\frac{1}{N}\sum_{n=0}^{N-1}U^{n}=E_{1}\,,
\end{equation}
$E_{1}$ being the selfadjoint projection onto the eigenspace of the 
invariant vectors for $U$. The entangled ergodic theorem in not yet available, and it is expected to fail
in the full generality (see
e.g.  pag. 8 of \cite{NSZ}). In addition, it is yet unknown what are 
general enough conditions under which it can be proved. In \cite{F1} 
it is shown that the entangled ergodic theorem holds true in the case 
when the $A_j$ in \eqref{0} are compact, without any condition on the unitary $U$, and in the almost periodic 
case (i.e. when $\ch$ is generated by the eigenvalues of $U$) for 
some very special pair partitions, without any condition on the 
$A_j$.  
Another interesting case arising from "quantum diagonal measures" is 
treated in \cite{F2}.

In the present note we prove that the entangled ergodic theorem holds 
true in the almost 
periodic case. Namely, the Cesaro mean in \eqref{0} converges in the 
strong operator topology for all the pair partitions $\a$, provided 
the dynamics generated by the unitary $U$ on the Hilbert space $\ch$ 
is almost periodic. We apply the present result to obtain the convergence of the Cesaro mean of several multiple correlations for $C^*$--dynamical systems such that the unitary implementing the dynamics in the GNS Hilbert space is almost periodic. 

For the sake of completeness, we report the analogous result involving the multiple correlations for $C^{*}$--dynamical systems based on compact operators.

\section{notations and basic facts}

Let $U\in\cb(\ch)$ be a unitary operator acting on the Hilbert space 
$\ch$. The unitary $U$ is said to be {\it almost periodic} if
$\ch=\ch_{\mathop{\rm ap}}^{U}$, $\ch_{\mathop{\rm ap}}^{U}$ being 
the closed subspace 
consisting of the vectors having relatively norm--compact orbit 
under $U$. It is seen in \cite{NSZ} that $U$ is almost periodic if 
and 
only if $\ch$ is generated by the eigenvectors of $U$. Denote $\s(U)$ 
and 
$\s_{\mathop{\rm pp}}(U)\subset \s(U)$ the spectrum and the pure 
point spectrum (i.e. set of all the eigenvalues of $U$) of $U$ 
respectively. Define
\begin{equation*}
\s_{\mathop{\rm pp}}^{\mathop{\rm a}}(U):=
\big\{z\in\s_{\mathop{\rm pp}}(U)\,\big|\,zw=1\,\text{for 
some}\,w\in\s_{\mathop{\rm pp}}(U)\big\}\,,
\end{equation*}
that is the "antidiagonal" part of $\s_{\mathop{\rm pp}}(U)$. A 
partition $\a:\{1,\dots,m\}\mapsto\{1,\dots, k\}$
of the set made of $m$ elements in $k$ parts is nothing but a 
surjective 
map, the parts of $\{1,\dots,m\}$ being the preimages 
$\{\a^{-1}(\{j\})\}_{j=1}^{k}$. A {\it pair partition} is nothing 
but a partition such that the preimages are made by two elements.

Consider, for each finite subset 
$F\subset\s_{\mathop{\rm pp}}^{\mathop{\rm a}}(U)$ and 
$\{A_{1},\dots,A_{2k-1}\}\subset\cb(\ch)$, the following 
operator
\begin{equation}
\label{opses}
S^{F}_{\a;A_{1},\dots,A_{2k-1}}:=
\sum_{z_{1},\dots,z_{k}\in 
F}E_{z^{\#}_{\a(1)}}A_{1}E_{z^{\#}_{\a(2)}}\cdots
E_{z^{\#}_{\a(2k-1)}}A_{2k-1}E_{z^{\#}_{\a(2k)}}
\end{equation}
together with the sesquilinear form
\begin{equation*}
s^{F}_{\a;A_{1},\dots,A_{2k-1}}(x,y):=\big\langle 
S^{F}_{\a;A_{1},\dots,A_{2k-1}}x,y\big\rangle\,,
\end{equation*}
where the pairs $z^{\#}_{\a(i)}$ are alternatively $z_{j}$ and
$\bar z_{j}$ whenever $\a(i)=j$, and $E_{z}$ is the selfadjoint 
projection on the eigenspace corresponding to the eigenvalue 
$z\in\s_{\mathop{\rm pp}}(U)$. If for example, $\a$ is the 
pair partition
$\{1,2,1,2\}$ of four elements, we write (cf. Proposition \ref{3}) for the limit in the weak 
operator topology of \eqref{opses},
$$
S_{\a;A,B,C}=
\sum_{z,w\in\s_{\mathop{\rm pp}}^{\mathop{\rm a}}(U)}E_{\bar 
z}AE_{\bar w}BE_{z}CE_{w}\,.
$$

The strong limits of the operators $S^{F}_{\a;A_{1},\dots,A_{2k-1}}$ will describe the limit of the Cesaro means \eqref{0} in the case under consideration in the present paper, see Theorem \ref{qper2} below. The reader is referred to \cite{F2} for the case when the $A_j$ are compact operators.

We report the proof of the following results for the convenience of 
the reader.
\begin{lem}
\label{1}
We have for the above sesquilinear form,
$$
\big|s^{F}_{\a;A_{1},\dots,A_{2k-1}}(x,y)\big|\leq\|x\|\|y\|
\prod_{j=1}^{2k-1}\|A_{j}\|\,,
$$
uniformly for $F$ finite subsets of $\s_{\mathop{\rm 
pp}}^{\mathop{\rm a}}(U)$.
\end{lem}
\begin{proof}
The proof follows by the repeated application of
$$
\big\|\sum_{j\in J}P_j\xi_j\big\|^2=\sum_{j\in J}\|P_j\xi_j\|^2\,,
$$
by taking into account the Schwarz and Bessel 
inequalities. Here, $\{P_j\}_{j\in J}$ is any orthogonal set of 
selfadjoint projections acting on a Hilbert space $\ch$, and 
$\{\xi_j\}_{j\in J}\subset\ch$. The reader is referred to \cite{F1} to see how the proof works in a pivotal case.
\end{proof}
\begin{lem}
\label{2}
The net $\big\{\sum_{z\in F}E_{\bar z}AE_{z}\,\big|\,F\,\text{finite 
subset of}\,\s_{\mathop{\rm pp}}^{\mathop{\rm a}}(U)\big\}$ 
converges in the strong operator 
topology.
\end{lem}
\begin{proof}
\begin{align*}
&\big\|\sum_{z\in F}E_{\bar z}AE_{z}x-\sum_{z\in G}E_{\bar 
z}AE_{z}x\big\|\\
\leq&\big\|\sum_{z\in F\backslash G}E_{\bar z}AE_{z}x\big\|
+\big\|\sum_{z\in G\backslash F}E_{\bar z}AE_{z}x\big\|\,.
\end{align*}
By taking into account Lemma \ref{1}, it is enough to 
prove that for
$\eps>0$, there exists a finite set $G_{\eps}$, such that
${\displaystyle\big\|\sum_{z\in H}E_{\bar 
z}AE_{z}x\big\|<\frac{\eps}{3}}$ whenever 
$H\subset G_{\eps}^{c}$. But,
$$
\big\|\sum_{z\in H}E_{\bar z}AE_{z}x\big\|^{2}
=\sum_{z\in H}\big\|E_{\bar z}AE_{z}x\big\|^{2}
\leq\|A\|^{2}\sum_{z\in H}\big\|E_{z}x\big\|^{2}\,.
$$
The proof follows as the last sum is convergent. 
\end{proof}
\begin{prop}
\label{3}
For each finite set $\{A_{1},\dots,A_{2k-1}\}\subset\cb(\ch)$, the 
net 
$\big\{S^{F}_{\a;A_{1},\dots,A_{2k-1}}\,\big|\,F\,\text{finite 
subset of}\,\s_{\mathop{\rm pp}}^{\mathop{\rm a}}(U)\big\}$ converges 
in the strong operator 
topology to a element in $\cb(\ch)$ denoted by  
$S_{\a;A_{1},\dots,A_{2k-1}}$.
\end{prop}
\begin{proof}
As for any finite $F\subset\s_{\mathop{\rm pp}}^{\mathop{\rm a}}(U)$, 
$$
S^{F}_{\a;A_{1},\dots,A_{2k-1}}=ES^{F}_{\a;A_{1},\dots,A_{2k-1}}E\,,
$$
$E$ being the selfadjoint projection onto the almost periodic subspace of $U$, we 
suppose without loss of generality, that 
$x\in\ch$ is an eigenvector of $U$ with eigenvalue $z_{0}$.
The proof is by induction on $k$. By Lemma \ref{2}, it is enough to 
show that the assertion holds for the pair partition 
$\b:\{1,\dots, 2k+2\}\mapsto\{1,\dots, k+1\}$, provided it is true 
for any pair partition $\a:\{1,\dots, 2k\}\mapsto\{1,\dots, k\}$. 
Let 
$k_{\b}\in\{1,\dots, 2k+2\}$ be the first element of the pair 
$\b^{-1}\big(\{k+1\}\big)$, and $\a_{\b}$ the pair partition of
$\{1,\dots, 2k\}$ obtained 
by deleting $\b^{-1}\big(\{k+1\}\big)$ from $\{1,\dots, 2k+2\}$, and 
$k+1$ from $\{1,\dots, k+1\}$. We obtain
$$
S^{F}_{\b;A_{1},\dots,A_{2k+1}}x
=S^{F}_{\a_{\b};A_{1},\dots,A_{k_{\b}-1}
E_{\bar z_{0}}A_{k_{\b}},\dots,A_{2k}}A_{2k+1}x\,,
$$
provided $\bar z_{0}\in F$.  We get
$$
\lim_{F\uparrow\s_{\mathop{\rm pp}}^{\mathop{\rm a}}(U)}
S^{F}_{\b;A_{1},\dots,A_{2k+1}}x=
S_{\a_{\b};A_{1},\dots,A_{k_{\b}-1}
E_{\bar z_{0}}A_{k_{\b}},\dots,A_{2k}}A_{2k+1}x\,,
$$
$S_{\a;A_{1},\dots,A_{2k-1}}$ being the limit in the strong operator 
topology of $S^{F}_{\a;A_{1},\dots,A_{2k-1}}$ which exists by the inductive 
hypothesis.
\end{proof}

We symbolically write
\begin{align}
\label{symb}
&S_{\a;A_{1},\dots,A_{2k-1}}:=
\mathop{\rm s\!-\!lim}_{F\uparrow\s_{\mathop{\rm pp}}^{\mathop{\rm 
a}}(U)}
S^{F}_{\a;A_{1},\dots,A_{2k-1}}\\
=&\sum_{z_{1},\dots,z_{k}\in\s_{\mathop{\rm pp}}^{\mathop{\rm a}}(U)}
E_{z^{\#}_{\a(1)}}A_{1}E_{z^{\#}_{\a(2)}}\cdots
E_{z^{\#}_{\a(2k-1)}}A_{2k-1}E_{z^{\#}_{\a(2k)}}\,,\nn
\end{align}
where in \eqref{symb} the pairs $z^{\#}_{\a(i)}$ are alternatively 
$z_{j}$ and
$\bar z_{j}$ whenever $\a(i)=j$ as in \eqref{opses}. By Lemma 
\ref{1}, we get
\begin{equation}
\label{symbbb}
\|S_{\a;A_{1},\dots,A_{2k-1}}\|\leq\prod_{j=1}^{2k-1}\|A_{j}\|\,.
\end{equation}

\section{the entangled ergodic theorem in the almost periodic case}

In the present section we prove the entangled ergodic theorem for the 
almost 
periodic situation. In this way, we improve the results 
in Section 3 of \cite{F1} where only very special pair partitions
were considered. 
We suppose 
that $\ch$ is generated by 
the eigenvectors of $U$ if it is not otherwise specified. 

The proof of the following result relies upon 
the mean ergodic theorem \eqref{jvn}, by 
showing that, step by step, one can reduce the matter to the dense 
subspace algebraically generated by the eigenvectors of $U$. 

We start by pointing out some preliminary facts on the 
pair partition 
$\a:\{1,2.\dots,2k\}\mapsto\{1,2.\dots,k\}$ used in the proof. We can 
put
$$
\{1,2.\dots,2k\}=\{i_1,i_2.\dots,i_k\}\bigcup\{j_k,\dots,j_2,j_1\}
$$
with $j_k<j_{k-1}<\cdots<j_2<j_1=2k$, $\a^{-1}(\{1\})=\{i_1,2k\}$, 
the order of the set $\{i_1,i_2.\dots,i_k\}$ is that determined by 
$\a$, $i_m$ is the greatest element of $\{i_1,i_2.\dots,i_k\}$ 
(perhaps possibly coinciding with $i_1$), and finally $j_h$ is the 
first element after $i_m$ (i.e. $i_m+1=j_h$. Namely, 
"$\bigcup$" stands for disjoint union, and 
$\a^{-1}(\{l\})=\{i_l,j_l\}$ with $i_l<j_l$, $l=1,2,\dots,k$.

\begin{thm}
\label{qper2}    
Let $U$ be an almost periodic unitary 
operator acting on the Hilbert space $\ch$. Then 
for each pair partition $\a:\{1,\dots,2k\}\mapsto\{1,\dots,k\}$,
and $A_{1},\dots,A_{2k-1}\in\cb(\ch)$,
\begin{align*}
\mathop{\rm 
s\!-\!lim}_{N\to+\infty}&\bigg\{\frac{1}{N^{k}}\sum_{n_{1},\dots,n_{k}=0}^{N-1}
U^{n_{\a(1)}}A_{1}U^{n_{\a(2)}}\cdots 
U^{n_{\a(2k-1)}}A_{2k-1}U^{n_{\a(2k)}}\bigg\}\\
=&S_{\a;A_{1},\dots,A_{2k-1}}\,.
\end{align*} 
\end{thm}
\begin{proof}
We suppose without loss of generality (cf. \eqref{symbbb}), that 
$\|A_i\|\leq1$, $i=1,\dots,2k-1$. 
Fix $\eps>0$, and choose recursively the following sets.
Let $I_{\eps}$ be such that 
$$
\bigg\|x-\sum_{\eta_1\in I_{\eps}}E_{\eta_1}x\bigg\|<\eps\,.
$$
For each $\eta_1\in I_{\eps}$, let $I_{\eps}(\eta_1)$ be such that 
$$
\bigg\|A_{2k-1}E_{\eta_1}x-\sum_{\eta_2\in 
I_{\eps}(\eta_1)}E_{\eta_2}A_{2k-1}E_{\eta_1}x\bigg\|<\frac{\eps}{|I_{\eps}|}\,,
$$
provided $i_1<j_2$.\footnote{We reduce the matter to this case as the partitions for which  
$\a(2k-1)=\a(2k)$ can be treated by taking into account that   
the product is jointly continuous in the strong operator topology 
when restricted to bounded parts.}
Finally, for each $\eta_1\in I_{\eps}$, $\eta_2\in 
I_{\eps}(\eta_1)$,..., $\eta_{k-1}\in 
I_{\eps}(\eta_1,\eta_2,\dots,\eta_{k-2})$, let 
$I_{\eps}(\eta_1,\eta_2,\dots,\eta_{k-2},\eta_{k-1})$ be such that
\begin{align*}
&\bigg\|A_{j_k+1}E_{\eta^\#_{\a(j_k+1)}}\cdots 
E_{\eta_2}A_{2k-1}E_{\eta_1}x\\
-&\sum_{\eta_k\in 
I_{\eps}(\eta_1,\eta_2,\dots,\eta_{k-2},\eta_{k-1})}E_{\eta_k}A_{j_k}
A_{j_k+1}E_{\eta^\#_{\a(j_k+1)}}\cdots 
E_{\eta_2}A_{2k-1}E_{\eta_1}x\bigg\|\\
<&\frac{\eps}
{\sum_{\eta_1\in I_{\eps}}\sum_{\eta_2\in I_{\eps}(\eta_1)}\cdots
\sum_{\eta_{k-1}\in I_{\eps}(\eta_1,\eta_2,\dots,\eta_{k-2})}
|I_{\eps}(\eta_1,\eta_2,\dots,\eta_{k-2},\eta_{k-1})|}\,.
\end{align*}
By taking into account \eqref{jvn}, choose $N_{\eps}$ such that 
\begin{align*}
&\bigg\|\bigg(\frac{1}{N}\sum_{n=0}^{N-1}(\eta_m 
U)^{n}-E_{\bar\eta_m}\bigg)
A_{i_m}E_{\eta_h}\cdots
E_{\eta_2}A_{2k-1}E_{\eta_1}x\bigg\|\\
<&\frac{\eps}
{\sum_{\eta_1\in I_{\eps}}\sum_{\eta_2\in I_{\eps}(\eta_1)}\cdots
\sum_{\eta_{h-1}\in I_{\eps}(\eta_1,\eta_2,\dots,\eta_{h-2})}
|I_{\eps}(\eta_1,\eta_2,\dots,\eta_{h-2},\eta_{h-1})|}\,,
\end{align*}
and after $k-1$ steps,
\begin{align*}
&\bigg\|\bigg(\frac{1}{N}\sum_{n=0}^{N-1}(\eta_k 
U)^{n}-E_{\bar\eta_k}\bigg)
A_{1}E_{\eta^\#_{\a(2)}}\cdots
E_{\eta_2}A_{2k-1}E_{\eta_1}x\bigg\|\\
<&\frac{\eps}
{\sum_{\eta_1\in I_{\eps}}\sum_{\eta_2\in I_{\eps}(\eta_1)}\cdots
\sum_{\eta_{k-1}\in I_{\eps}(\eta_1,\eta_2,\dots,\eta_{k-2})}
|I_{\eps}(\eta_1,\eta_2,\dots,\eta_{k-2},\eta_{k-1})|}\,,
\end{align*}
whenever $N>N_{\eps}$. We then have
\begin{align*}
&\bigg\|\bigg(\frac{1}{N^{k}}\sum_{n_{1},\dots,n_{k}=0}^{N-1}
U^{n_{\a(1)}}A_{1}U^{n_{\a(2)}}\cdots 
U^{n_{\a(2k-1)}}A_{2k-1}U^{n_{\a(2k)}}
-S_{\a;A_{1},\dots,A_{2k-1}}\bigg)x\bigg\|\\
\leq&\bigg\|\bigg(\frac{1}{N^{k}}\sum_{n_{1},\dots,n_{k}=0}^{N-1}
U^{n_{\a(1)}}A_{1}U^{n_{\a(2)}}\cdots 
U^{n_{\a(2k-1)}}A_{2k-1}U^{n_{\a(2k)}}
-S_{\a;A_{1},\dots,A_{2k-1}}\bigg)\bigg\|\\
\times&\bigg\|\bigg(x-\sum_{\eta_1\in 
I_{\eps}}E_{\eta_1}x\bigg)\bigg\|
\end{align*}
\begin{align*}
+&\bigg\|\sum_{\eta_1\in 
I_{\eps}}\bigg(\frac{1}{N^{k}}\sum_{n_{1},\dots,n_{k}=0}^{N-1}
U^{n_{\a(1)}}A_{1}U^{n_{\a(2)}}\cdots 
A_{i_1-1}(\eta_1U)^{n_{1}}A_{i_1}\cdots U^{n_{\a(2k-1)}}\\
-&\sum_{z_{2},\dots,z_{k}\in\s_{\mathop{\rm pp}}^{\mathop{\rm 
a}}(U)}
E_{z^{\#}_{\a(1)}}A_{1}E_{z^{\#}_{\a(2)}}\cdots
A_{i_1-1}E_{\bar\eta_1}A_{i_1}\cdots 
E_{z^{\#}_{\a(2k-1)}}\bigg)A_{2k-1}E_{\eta_1}x\bigg\|\\
\leq&2\eps+\bigg\|\sum_{\eta_1\in 
I_{\eps}}\bigg(\frac{1}{N^{k}}\sum_{n_{1},\dots,n_{k}=0}^{N-1}
U^{n_{\a(1)}}A_{1}U^{n_{\a(2)}}\cdots 
A_{i_1-1}(\eta_1U)^{n_{1}}A_{i_1}\cdots U^{n_{\a(2k-1)}}
\end{align*}
\begin{align*}
-&\sum_{z_{2},\dots,z_{k}\in\s_{\mathop{\rm pp}}^{\mathop{\rm 
a}}(U)}
E_{z^{\#}_{\a(1)}}A_{1}E_{z^{\#}_{\a(2)}}\cdots
A_{i_1-1}E_{\bar\eta_1}A_{i_1}\cdots 
E_{z^{\#}_{\a(2k-1)}}\bigg)\bigg\|\\
\times&\bigg\|\bigg(A_{2k-1}E_{\eta_1}x-\sum_{\eta_2\in 
I_{\eps}(\eta_1)}
E_{\eta_2}A_{2k-1}E_{\eta_1}x\bigg)\bigg\|\\
+&\bigg\|\sum_{\eta_1\in I_{\eps}}\sum_{\eta_2\in I_{\eps}(\eta_1)}
\bigg(\frac{1}{N^{k}}\sum_{n_{1},\dots,n_{k}=0}^{N-1}
U^{n_{\a(1)}}A_{1}U^{n_{\a(2)}}\cdots 
A_{i_1-1}(\eta_1U)^{n_{1}}A_{i_1}\cdots U^{n_{\a(2k-1)}}
\end{align*}
\begin{align*}
-&\sum_{z_{2},\dots,z_{k}\in\s_{\mathop{\rm pp}}^{\mathop{\rm 
a}}(U)}
E_{z^{\#}_{\a(1)}}A_{1}E_{z^{\#}_{\a(2)}}\cdots
A_{i_1-1}E_{\bar\eta_1}A_{i_1}\cdots 
A_{2k-2}\bigg)E_{\eta_2}A_{2k-1}E_{\eta_1}x\bigg\|\\
&\qquad\qquad\vdots\qquad\vdots\qquad\vdots\qquad\qquad\vdots\qquad\vdots\qquad
\vdots\qquad\qquad\vdots\qquad\vdots\qquad\vdots\qquad\qquad\\
\leq&2(2k-i_m)\eps+\bigg\|\sum_{\eta_1\in I_{\eps}}\sum_{\eta_2\in 
I_{\eps}(\eta_1)}\cdots
\sum_{\eta_{h}\in 
I_{\eps}(\eta_1,\eta_2,\dots,\eta_{h-1})}\frac{1}{N^{k-1}}
\end{align*}
\begin{align*}
\times&\sum_{n_{1},\dots,n_{m-1},n_{m+1},\dots n_{k}=0}^{N-1}
U^{n_{\a(1)}}A_{1}U^{n_{\a(2)}}\cdots A_{i_m-1}
\bigg(\frac{1}{N}\sum_{n=0}^{N-1}(\eta_m U)^{n}-E_{\bar\eta_m}\bigg)\\
\times&A_{i_m}E_{\eta_h}\cdots E_{\eta_2}A_{2k-1}E_{\eta_1}x\bigg\|
+\bigg\|\sum_{\eta_1\in I_{\eps}}\sum_{\eta_2\in 
I_{\eps}(\eta_1)}\cdots
\sum_{\eta_{h}\in I_{\eps}(\eta_1,\eta_2,\dots,\eta_{h-1})}\\
\times&\bigg(\frac{1}{N^{k-1}}
\sum_{n_{1},\dots,n_{m-1},n_{m+1},\dots n_{k}=0}^{N-1}
U^{n_{\a(1)}}A_{1}U^{n_{\a(2)}}\cdots U^{n_{\a(i_m-1)}}
\end{align*}
\begin{align*}
-&\sum_{z_{h+1},\dots,z_{k}\in\s_{\mathop{\rm pp}}^{\mathop{\rm a}}(U)}
E_{z^{\#}_{\a(1)}}A_{1}E_{z^{\#}_{\a(2)}}\cdots 
E_{z^{\#}_{\a(i_m-1)}}\bigg)
A_{i_m-1}E_{\bar\eta_1}A_{i_m}\cdots 
E_{\eta_2}A_{2k-1}E_{\eta_1}x\bigg\|\\
\leq&[2(2k-i_m)+1]\eps+\bigg\|\sum_{\eta_1\in 
I_{\eps}}\sum_{\eta_2\in I_{\eps}(\eta_1)}\cdots
\sum_{\eta_{h}\in I_{\eps}(\eta_1,\eta_2,\dots,\eta_{h-1})}\\
\times&\bigg(\frac{1}{N^{k-1}}
\sum_{n_{1},\dots,n_{m-1},n_{m+1},\dots n_{k}=0}^{N-1}
U^{n_{\a(1)}}A_{1}U^{n_{\a(2)}}\cdots U^{n_{\a(i_m-1)}}
\end{align*}
\begin{align*}
-&\sum_{z_{h+1},\dots,z_{k}\in\s_{\mathop{\rm pp}}^{\mathop{\rm a}}(U)}
E_{z^{\#}_{\a(1)}}A_{1}E_{z^{\#}_{\a(2)}}\cdots 
E_{z^{\#}_{\a(i_m-1)}}\bigg)
A_{i_m-1}E_{\bar\eta_1}A_{i_m}\cdots 
E_{\eta_2}A_{2k-1}E_{\eta_1}x\bigg\|\\
&\qquad\qquad\vdots\qquad\vdots\qquad\vdots\qquad\qquad\vdots\qquad\vdots\qquad\vdots\qquad\qquad\vdots
\qquad\vdots\qquad\vdots\qquad\qquad\\
\leq&(3k-1)\eps+\bigg\|\sum_{\eta_1\in I_{\eps}}\sum_{\eta_2\in 
I_{\eps}(\eta_1)}\cdots
\sum_{\eta_{k}\in I_{\eps}(\eta_1,\eta_2,\dots,\eta_{k-1})}
\bigg(\frac{1}{N}\sum_{n=0}^{N-1}(\eta_k U)^{n}-E_{\bar\eta_k}\bigg)\\
\times&A_{1}E_{\eta^\#_{\a(2)}}\cdots
E_{\eta_2}A_{2k-1}E_{\eta_1}x\bigg\|\leq3k\eps\,.
\end{align*}
\end{proof}

\section{multiple correlations}
\label{coorre}

The study of multiple correlations is a standard matter of interest in classical and quantum ergodic theory for several application to various fields. For example they are of interest to investigate the chaotic behavior of dynamical systems. We also mention the natural applications to quantum statistical mechanics, number theory, probability. The reader is referred to \cite{Fu, NSZ} for further details (see also \cite{F2, F3} for some partial results involving multiple correlations and recurrence). The present analysis allows us to study the limit of the Cesaro mean of several multiple correlations. 

We start with a $C^*$--dynamical system $(\ga,\g,\om)$ made of a $C^*$--algebra $\ga$, an automorphism $\g$ of $\ga$, and finally a state $\om$ on $\ga$ which is invariant under $\g$. Consider the covariant GNS representation $(\pi_{\om}, \ch_{\om}, \Om, U)$ (cf. \cite{T}) associated to the 
$C^*$--dynamical system under consideration. 
\begin{thm}
\label{rree}
Under the above notations, suppose that $U$ implementing $\g$ on $\ch_{\om}$ is almost periodic. Then for each pair partition $\a:\{1,\dots, 2k\}\mapsto\{1,\dots, k\}$, we have
\begin{align*}
&\lim_N\frac{1}{N^{k}}\sum_{n_{1},\dots,n_{k}=0}^{N-1}
\om\big(A_0\g^{n_{\a(1)}}(A_{1})\\
\times&\g^{n_{\a(1)}+n_{\a(2)}}(A_2)\cdots 
\g^{(\sum_{l=1}^{2k-1}n_{\a(l)})}(A_{2k-1})
\g^{(\sum_{i=1}^{k}2n_{i})}(A_{2k})\big)\\
=&\big\langle\pi_{\om}(A_0)S_{\a;\pi_{\om}(A_1),\dots,\pi_{\om}(A_{2k-1})}\pi_{\om}(A_{2k})
\Om,\Om\big\rangle\,.
\end{align*}
\end{thm}
\begin{proof}
The proof directly follows from Theorem \ref{qper2}, by taking into account that
\begin{align*}
&\om\big(A_0\g^{n_{\a(1)}}(A_{1})
\g^{n_{\a(1)}+n_{\a(2)}}(A_2)\cdots 
\g^{(\sum_{l=1}^{2k-1}n_{\a(l)})}(A_{2k-1})
\g^{(\sum_{i=1}^{k}2n_{i})}(A_{2k})\big)\\
=&\big\langle\pi_{\om}(A_0)U^{n_{\a(1)}}\pi_{\om}(A_{1})U^{n_{\a(2)}}\cdots 
U^{n_{\a(2k-1)}}\pi_{\om}(A_{2k-1})U^{n_{\a(2k)}}\pi_{\om}(A_{2k})\Om,\Om\big\rangle\,.
\end{align*}
\end{proof}

For the sake of completeness, we notice that the same result holds true for $C^{*}$--dynamical systems of compact operators, without any condition on the unitary $U$ implementing the dynamics. Namely, let 
$(\ck(\ch),\g)$ a $C^{*}$--dynamical system based on the algebra of all the compact operators 
acting on the Hilbert space $\ch$. It is well known that $\g=\ad(U)$, that is it is unitarily implemented on 
$\ch$ by the adjoint action of some unitary $U$, uniquely determined up to a phase. We denote by $\g$ the adjoint action $\ad(U)\equiv\g^{**}$ on $\cb(\ch)$ as well.\footnote{It is enough to consider the double transpose 
$\g^{**}\in\aut(\cb(\ch))$, which is an automorphism of $\cb(\ch)$, and therefore inner. The previous mentioned phase factor is inessential for our computations.} 

Let $\tr$ be the canonical trace on $\cb(\ch)$, and consider a positive normalized trace class operator 
$T$
acting on $\ch$ such that $s(T)\leq E_1$, $s(T)$ and $E_1$ being the support of $T$ and the spectral projection onto the invariant vectors for $U$, respectively. It is easy to show that
\begin{equation}
\label{tinv}
UT=T=TU\,.
\end{equation}
Let $\om_T\in\cb(\ch)_*$ be the state defined as
$$
\om_T(A):=\tr(TA)\,,\quad A\in\cb(\ch)\,.
$$

Thanks to \eqref{tinv}, it is invariant under $\g$, as its restriction to $\ck(\ch)$ denoted with an abuse of notation also by $\om_T$. The following result parallels Theorem \ref{rree}.
\begin{thm}
For each pair partition $\a:\{1,\dots, 2k\}\mapsto\{1,\dots, k\}$,  $\{A_0, A_{2k}\}\subset\cb(\ch)$ and 
$\{A_1,\dots A_{2k-1}\}\subset\ck(\ch)$, we have
\begin{align}
\label{limi}
&\lim_N\frac{1}{N^{k}}\sum_{n_{1},\dots,n_{k}=0}^{N-1}
\om_T\big(A_0\g^{n_{\a(1)}}(A_{1})
\g^{n_{\a(1)}+n_{\a(2)}}(A_2)\times\cdots\nn\\ 
\times&\g^{(\sum_{l=1}^{2k-1}n_{\a(l)})}(A_{2k-1})
\g^{(\sum_{i=1}^{k}2n_{i})}(A_{2k})\big)\\
=&\om_T\big(A_0S_{\a;A_1,\dots,A_{2k-1}}A_{2k}\big)\nn\,.
\end{align}
\end{thm}
\begin{proof}
By \eqref{tinv}, we get
\begin{align*}
&\om_T\big(A_0\g^{n_{\a(1)}}(A_{1})
\g^{n_{\a(1)}+n_{\a(2)}}(A_2)\cdots 
\g^{(\sum_{l=1}^{2k-1}n_{\a(l)})}(A_{2k-1})
\g^{(\sum_{i=1}^{k}2n_{i})}(A_{2k})\big)\\
=&\om_T\big(A_0U^{n_{\a(1)}}A_{1}U^{n_{\a(2)}}\cdots 
U^{n_{\a(2k-1)}}A_{2k-1}U^{n_{\a(2k)}}A_{2k}\big)\,.
\end{align*}

The proof directly follows from Theorem 2.6 of \cite{F1} by approximating the trace class operator $T$ by a finite rank one.
\end{proof}

We conclude by noticing that Theorem 2.6 of \cite{F1} allows us to treat all the multiple correlations arising from any general partition of any set of $m$ points in $k$ parts, for dynamical systems based on the compact operators. In the case of non pair partitions, it is not immediate to provide a general formula for the limit in \eqref{limi}.

\section{appendix}

Unfortunately, a proof of our main theorem (or of the estimation in Lemma \ref{1}) based on the induction principle works well only for non crossing partitions.\footnote{Among the set of pair partitions of four elements, $\{1,2,2,1\}$ and $\{1,1,2,2\}$ are non crossing, whereas the remaining one $\{1,2,1,2\}$ is crossing. The reader is referred to \cite{AHO} for the abstract definition of crossing partitions.}
Due to entanglement and to the fact that the mean ergodic theorem \eqref{jvn} holds true only in the strong operator topology, any attempt to provide any kind of induction proof of Theorem \ref{qper2} produces essentially the same complexity as the proof presented in this paper. While keeping the original proof, to show how the last is working, we report the particular case of the entangled partition $\a=\{1,2,1,3,2,3\}$. 

Fix $\eps>0$, and suppose that $A,B,C,D,F\in\cb(\ch)$ have norm  
one. Let $I_{\eps}$ be such that 
$$
\bigg\|x-\sum_{\s\in I_{\eps}}E_{\s}x\bigg\|<\eps\,.
$$
For each $\s\in I_{\eps}$, let $I_{\eps}(\s)$ be such that 
$$
\bigg\|FE_{\s}x-\sum_{\t\in 
I_{\eps}(\s)}E_{\t}FE_{\s}x\bigg\|<\frac{\eps}{|I_{\eps}|}\,.
$$
Finally, for each $\s\in I_{\eps}$, $\t\in I_{\eps}(\s)$, let $I_{\eps}(\s,\t)$
be such that 
$$
\bigg\|CE_{\bar\s}DE_{\t}FE_{\s}x-\sum_{\r\in I_{\eps}(\s,\t)}E_{\r}C
E_{\bar\s}DE_{\t}FE_{\s}x\bigg\|<\frac{\eps}
{{\displaystyle\sum_{\s\in I_{\eps}}|I_{\eps}(\s)|}}\,.
$$
In addition, by the mean ergodic theorem \eqref{jvn}, choose $N_{\eps}$ such that 
\begin{align*}
&\bigg\|\bigg(\frac{1}{N}\sum_{n=0}^{N-1}(\s U)^{n}-E_{\bar\s}\bigg)
DE_{\t}FE_{\s}x\bigg\|<\frac{\eps}{{\displaystyle\sum_{\s\in I_{\eps}}|I_{\eps}(\s)|}}\,,\\
&\bigg\|\bigg(\frac{1}{N}\sum_{n=0}^{N-1}(\t U)^{n}-E_{\bar\t}\bigg)
BE_{\r}CE_{\bar\s}DE_{\t}FE_{\s}x\bigg\|<\frac{\eps}{{\displaystyle\sum_{\s\in I_{\eps}}
\sum_{\t\in I_{\eps}(\s)}|I_{\eps}(\s,\t)|}}\,,\\
&\bigg\|\bigg(\frac{1}{N}\sum_{n=0}^{N-1}(\r U)^{n}-E_{\bar\r}\bigg)
AE_{\bar\t}BE_{\r}CE_{\bar\s}DE_{\t}FE_{\s}x\bigg\|<\frac{\eps}{{\displaystyle\sum_{\s\in I_{\eps}}
\sum_{\t\in I_{\eps}(\s)}|I_{\eps}(\s,\t)|}}\,,
\end{align*}
whenever $N>N_{\eps}$, $\s\in I_{\eps}$, $\t\in I_{\eps}(\s)$, and $\r\in I_{\eps}(\s,\t)$.  
Let $\b=\{1,2,1,2\}$, and $\g=\{1,1\}$.
We then obtain for each $N>N_{\eps}$,
\begin{align*}
&\bigg\|\frac{1}{N^{3}}\sum_{k,m,n=0}^{N-1}U^{k}AU^{m}BU^{k}CU^{n}DU^{m}FU^{n}x
-S_{\a;A,B,C,D,F}x\bigg\|\\
\leq&\bigg\|\frac{1}{N^{3}}\sum_{k,m,n=0}^{N-1}U^{k}AU^{m}BU^{k}CU^{n}DU^{m}FU^{n}
-S_{\a;A,B,C,D,F}\bigg\|\bigg\|x-\sum_{\s\in I_{\eps}}E_{\s}x\bigg\|\\
+&\bigg\|\sum_{\s\in I_{\eps}}
\bigg(\frac{1}{N^{3}}\sum_{k,m,n=0}^{N-1}U^{k}AU^{m}BU^{k}C(\s U)^{n}D U^{m}
-S_{\b;A,B,CE_{\bar\s}D}\bigg)FE_{\s}x\bigg\|\\
\end{align*}
\begin{align*}
\leq&2\eps+\bigg\|\frac{1}{N^{3}}\sum_{k,m,n=0}^{N-1}U^{k}AU^{m}BU^{k}C(\s U)^{n}DU^{m}
-S_{\b;A,B,CE_{\bar\s}D}\bigg\|\\
\times&\sum_{\s\in I_{\eps}}\bigg\|
FE_{\s}x-\sum_{\t\in I_{\eps}(\s)}E_{\t}FE_{\s}x\bigg\|\\
+&\bigg\|\sum_{\s\in I_{\eps}}
\sum_{\t\in I_{\eps}(\s)}
\bigg(\frac{1}{N^{3}}\sum_{k,m,n=0}^{N-1}U^{k}A(\t U)^{m}BU^{k}C(\s U)^{n}
-S_{\g;AE_{\bar\t}B}CE_{\bar\s}\bigg)DE_{\t}FE_{\s}x\bigg\|\\
\leq&4\eps+\bigg\|\sum_{\s\in I_{\eps}}\sum_{\t\in I_{\eps}(\s)}
\bigg[\frac{1}{N^{2}}\sum_{k,m=0}^{N-1}U^{k}A(\t U)^{m}BU^{k}C\bigg(\frac{1}{N}\sum_{n=0}^{N-1}
(\s U)^{n}-E_{\bar\s}\bigg)\bigg]DE_{\t}FE_{\s}x\bigg\|\\
+&\bigg\|\sum_{\s\in I_{\eps}}
\sum_{\t\in I_{\eps}(\s)}
\bigg(\frac{1}{N^{2}}\sum_{k,m=0}^{N-1}U^{k}A(\t U)^{m}BU^{k}
-S_{\g;AE_{\bar\t}B}\bigg)CE_{\bar\s}DE_{\t}FE_{\s}x\bigg\|\\
\leq&5\eps+\bigg\|\frac{1}{N^{2}}\sum_{k,m=0}^{N-1}U^{k}A(\t U)^{m}BU^{k}
-S_{\g;AE_{\bar\t}B}\bigg\|\\
\times&\bigg\|CE_{\bar\s}DE_{\t}FE_{\s}x-\sum_{\r\in I_{\eps}(\s,\t)}E_{\r}C
E_{\bar\s}DE_{\t}FE_{\s}x\bigg\|\\
+&\bigg\|\sum_{\s\in I_{\eps}}\sum_{\t\in I_{\eps}(\s)}\sum_{\r\in I_{\eps}(\s,\t)}
\bigg(\frac{1}{N^2}\sum_{k,m=0}^{N-1}(\r U)^{k}A(\t U)^{m}-E_{\bar\r}AE_{\bar\t}\bigg)
BE_{\r}C
E_{\bar\s}DE_{\t}FE_{\s}x\bigg\|\\
\leq&7\eps+\bigg\|\sum_{\s\in I_{\eps}}\sum_{\t\in I_{\eps}(\s)}\sum_{\r\in I_{\eps}(\s,\t)}
\bigg[\frac{1}{N}\sum_{k=0}^{N-1}(\r U)^{k}A\bigg(\frac{1}{N}
\sum_{m=0}^{N-1}(\t U)^{m}-E_{\bar\t}\bigg)\bigg]
BE_{\r}C
E_{\bar\s}DE_{\t}FE_{\s}x\bigg\|\\
+&\bigg\|\sum_{\s\in I_{\eps}}\sum_{\t\in I_{\eps}(\s)}\sum_{\r\in I_{\eps}(\s,\t)}
\bigg(\frac{1}{N}\sum_{n=0}^{N-1}(\r U)^{n}-E_{\bar\r}\bigg)
AE_{\bar\t}BE_{\r}CE_{\bar\s}DE_{\t}FE_{\s}x\bigg\|\leq9\eps\,.
\end{align*}

\section*{acknowledgements}

The author would like to thank F. Mukhamedov for some useful suggestions.

\end{document}